\documentclass[12pt,reqno]{amsart}
\usepackage{graphicx}
\usepackage{psfrag}
\usepackage{mathrsfs}
\usepackage{color}
\usepackage{amsmath,amsthm,amsfonts,amssymb,amscd}
\font\de=cmssi12

\numberwithin{equation}{section}

\newcommand{\diff}{\operatorname{Diff}}

\newcommand{\supp}{\operatorname{supp}}

\newcommand{\s} {\sigma}

\def \cW {{\mathcal W}}

\newcommand \W {\cW}
\newtheorem{theorem}{Theorem}
\newtheorem{remark}{Remark}
\newtheorem{question}{Question}
\newtheorem{corollary}{Corollary}
\newtheorem{lemma}{Lemma}
\newtheorem{proposition}{Proposition}

\begin{document}

\thanks{ }

\author{F. Rodriguez Hertz}
\address{IMERL-Facultad de Ingenier\'\i a\\ Universidad de la
Rep\'ublica\\ CC 30 Montevideo, Uruguay.}
\email{frhertz@fing.edu.uy}\urladdr{http://www.fing.edu.uy/$\sim$frhertz}

\author{M. A. Rodriguez Hertz}
\address{IMERL-Facultad de Ingenier\'\i a\\ Universidad de la
Rep\'ublica\\ CC 30 Montevideo, Uruguay.}
\email{jana@fing.edu.uy}\urladdr{http://www.fing.edu.uy/$\sim$jana}

\author{A. Tahzibi}
\address{Departamento de Matem\'atica,
  ICMC-USP S\~{a}o Carlos, Caixa Postal 668, 13560-970 S\~{a}o
  Carlos-SP, Brazil.}
\email{tahzibi@icmc.sc.usp.br}\urladdr{http://www.icmc.sc.usp.br/$\sim$tahzibi}

\author{R. Ures}
\address{IMERL-Facultad de Ingenier\'\i a\\ Universidad de la
Rep\'ublica\\ CC 30 Montevideo, Uruguay.} \email{ures@fing.edu.uy}
\urladdr{http://www.fing.edu.uy/$\sim$ures}

\thanks{A. Tahzibi was supported by CNPq (post-doctoral  fellowship at Facultad de Ingenier\'\i a of Universidad de la Rep\'{u}blica del Uruguay) and Fapesp.}



\renewcommand{\subjclassname}{\textup{2000} Mathematics Subject Classification}


\setcounter{tocdepth}{2}

\title[Maximizing measures for partially hyperbolic systems]{Maximizing measures for partially hyperbolic systems with compact center leaves}

\begin{abstract} We obtain the following dichotomy for accessible partially hyperbolic diffeomorphisms of 3-dimensional manifolds having compact center leaves: either there is a unique entropy maximizing measure, this measure has the Bernoulli property and its center Lyapunov exponent is 0 or, there is a finite number of entropy maximizing measures, all of them with nonzero center Lyapunov exponent (at least one with negative exponent and one with positive exponent), that are finite extensions of a Bernoulli system. In the first case of the dichotomy we obtain that the system is topologically conjugated to a rotation extension of a hyperbolic system. This implies that the second case of the dichotomy holds for an open and dense set of diffeomorphisms in the hypothesis of our result. As a consequence we obtain an open set of topologically mixing diffeomorphisms having more than one entropy maximizing measure.


\end{abstract}

\maketitle

\section{Introduction}
Topological and metric entropy are among the most important invariants for measuring  the degree of unpredictability in a dynamical system  on an exponential scale. By variational principle for homeomorphisms defined on compact spaces, the topological entropy is the supremum over the metric entropies of all invariant probability measures. The study of invariant measures that attain the highest entropy is a natural way of describing the behavior of most of the relevant orbits of a system.

Uniformly hyperbolic dynamical systems admit entropy maximizing measures and topological transitivity implies the uniqueness of such measures. The construction of maximizing measures for hyperbolic systems is due to R. Bowen \cite{bowen} and G. Margulis \cite {margulis} in two different approaches. By means of their proofs it comes out that maximizing measures equidistribute periodic orbits of the dynamics and also have a local product structure.

 Beyond uniform hyperbolicity even the existence problem in general is not clear. On the one hand,  M. Misiurewicz \cite{misi} has constructed a class of examples of diffeomorphisms without any entropy maximizing measure. On the other hand, the regularity of the dynamics is  relevant to the regularity of function which attributes the metric entropy to each invariant measure and by a result of S. Newhouse \cite{newhouse} this is an upper semicontinuous function in the case of $C^{\infty}$ diffeomorphisms of compact manifolds. Hence, any $C^{\infty}$ diffeomorphism admit entropy maximizing measures. Howewer, if we are in the setting of  partially hyperbolic diffeomorphism,  W. Cowieson and L.-S. Young results in \cite{cowiyoung} (see also \cite{diaz-fisher}) imply that there are always entropy maximizing measures if the center bundle is one-dimensional even if the diffeomorphism is $C^1$.

In this paper we address the problem of existence and uniqueness (finiteness) of entropy maximizing measures in the partially hyperbolic context. Although, as we have already mentioned above, it is known, we give a proof of the existence  because it is almost straightforward in our setting.

Let us mention some previous related works. In \cite{newhouse-young} S. Newhouse and L. S. Young have shown that some partially hyperbolic examples of $\mathbb{T}^4$  have a unique entropy maximizing measure (see for instance the examples in \cite{shub}). In 3-dimensional manifolds,  J. Buzzi, T. Fisher, M. Sambarino and C. V\'{a}squez \cite{bfsv} proved that certain DA construction due to R. Ma\~{n}\'{e} (see \cite{mane}) also has this uniqueness property. Moreover, thank to new techniques recently developed by M. Brin, D. Burago and Ivanov \cite{brin-burago-ivanov} and A. Hammerlindl \cite{hammerlindl1} in the theory of partially hyperbolic dynamics,  R. Ures \cite{ures} announced a proof of the same property for any (absolutely) partially hyperbolic diffeomorphism of the 3-torus homotopic to a hyperbolic automorphism.

\begin{theorem} \label{dichotomy}
Let $f : M \rightarrow M$ be a partially hyperbolic diffeomorphism, dynamically coherent with compact one dimensional central leaves and satisfying accessibility property. Then one and only one of the following occurs
\begin{enumerate}
\item $f$ admits a unique entropy maximizing measure $\mu$ and $\lambda_c(\mu) = 0$. Moreover $(f, \mu)$ is isomorphic to a Bernoulli shift,

\item  $f$ has a finite number (strictly greater than one) of ergodic maximizing measure all of which with non vanishing central Lyapunov exponent. Moreover $(f, \mu)$ is a finite extension of Bernoulli shift for any entropy maximizing measure $\mu.$ \end{enumerate}
Moreover the diffeomorphisms fulfilling the conditions of the second item form a $C^1-$open and $C^{\infty}-$dense subset of the dynamically coherent partially hyperbolic diffeomorphism with compact one dimensional central leaves.
\end{theorem}

In the proof of this theorem, we consider partially hyperbolic diffeomorphisms with compact central leaves as a non linear cocycle over a uniformly hyperbolic homeomorphism and apply an invariance principle proved by A. Avila and M. Viana \cite{AV}.

Let us mention that in the first item of the above theorem, based in ideas similar to the ones of A. Avila, M. Viana and A. Wilkinson in \cite{AVW}, the dynamics will be shown to be conjugate to an isometric extension of a hyperbolic homeomorphism  implying that the subset of dynamics satisfying the first item  is meager.  Observe that, as a corollary we obtain that for an open and dense subset of partially hyperbolic diffeomorphisms satisfying the conditions of the above theorem the ergodic maximizing measures are hyperbolic, i.e. the Lyapunov exponents are non-vanishing.

In \cite{newhouse-young} S. Newhouse and L. S. Young introduced the  definition of almost conjugacy. Let us recall that a system is intrinsically ergodic if it has a unique entropy maximizing measure (see \cite{weiss}). Let $f:X\rightarrow X$ and $g:Y\rightarrow Y$  be intrinsically ergodic systems (that is having a unique entropy maximizing measure) with maximal measures $\mu$ and $\overline\mu$, respectively. Say that $f$ and $g$ are almost conjugated if there are invariant sets $A\subset X$ and  $B\subset Y$ such that $\mu(A)=\overline\mu (B)=1$ and $f|_{A}$ is topologically conjugated to $g|_{B}$. They formulated the following question:
let $\mathcal{B}(M)$ denote the set of $C^r$ diffeomorphisms $f$ such that
\begin{enumerate}
\item $f $ has only finitely many ergodic measures of maximal entropy.
\item On each support of an ergodic measure of maximal entropy, $f$ is almost conjugate to the restriction of an Axiom A diffeomorphism to a topologically transitive basic set.
\end{enumerate}
Is $\mathcal{B}(M)$ residual in $\diff^r(M)$?

 Observe that a positive answer to Newhouse-Young question would imply that generically the support of entropy maximizing measures will not coincide.

Using Theorem \ref{dichotomy} and the robustly (topologically) mixing diffeomorphisms obtained by C. Bonatti and L. J. D\'{\i}az \cite{bonatti-diaz} we get examples of robustly mixing systems with more than one maximizing measures. This contrasts with the uniformly hyperbolic case and gives a negative answer to a question of J. Buzzi
 and T. Fisher \cite{bfsv}.

\begin{theorem}\label{mixing}
 There exist robustly mixing diffeomorphisms where the number of entropy maximizing measures is larger than one.
\end{theorem}

  Here a natural question arise:

  \begin{question}

  Is the support of any of the measures given by  Theorem \ref{mixing} the whole manifold?
  \end{question}

  A positive answer to this question would imply the existence of an open set of diffeomorphisms having entropy maximizing measures with full support giving a negative answer to Newhouse-Young question. Although this is the most likely situation,  we obtain that in our setting the answer will be positive if we modify slightly the definition of almost conjugacy.
 \begin{corollary} \label{answer-ny}

 Let $\mathcal{C}$ be the set of partially hyperbolic diffeomorphism of 3-dimensional manifolds that are dynamically coherent with compact one dimensional central leaves. Then, there exists an open and dense subset $\tilde{\mathcal{C}}$  of $\mathcal{C}$ such that for $f\in\tilde{\mathcal{C}}$ we have
 \begin{enumerate}
 \item $f$ has a finite number of entropy maximizing measures.
 \item If $\mu$ is an entropy maximizing measure for $f$, there exist a transitive hyperbolic basic set of an Axiom A diffeomorphism $g$ with entropy maximizing measure $m$, a subset $A\subset \supp (\mu)$  with $\mu(A)=1$ and a subset $B\subset \supp(m)$ with $m(B)=1$ such that $f|_A$ is topologically conjugated to $g|_B$.
 \end{enumerate}
 \end{corollary}

 We observe that the only difference with Newhouse-Young question is that we do not require $\mu $ to be the unique measure maximizing entropy of $f$ restricted to its support.

We emphasize that in the category of smooth invariant measures $\mu,$ by Pesin's result, if all the Lyapunov exponents are non zero and $(f^n, \mu)$ is ergodic for any $n \geq 1$ then $(f, \mu)$ is Bernoulli. Here we  see that the entropy maximizing measure with zero center Lyapunov exponent has the Bernoulli property by using profound results of D. Rudolph \cite{Rudolph} about isometric extensions of Bernoulli shifts.

The paper is organized as follows. In Section \ref{section.prelim.dinam} we present some preliminaries and definitions needed for the understanding of the paper. In Section \ref{quotient} we describe the space of center leaves and, moreover, we show that the manifold $M$ is (finitely covered by) a nilmanifold. In Section \ref{existence} we give a proof of the existence of  entropy maximizing measures and we show that any such a measure projects onto a measure with local product structure. In Section \ref{rigidity} we prove the first part of the dichotomy of Theorem \ref{dichotomy} and in Section \ref{parte2} we prove the second part. In Section \ref{section.mixing} we prove Theorem \ref{mixing} and Corollary \ref{answer-ny}. Finally, in Section \ref{questions} the reader can find some questions and directions of development.


\section{Dynamic preliminaries}\label{section.prelim.dinam}

\subsection{Partial Hyperbolicity} Throughout this paper we shall work with a {\de
partially hyperbolic diffeomorphism} $f$, that is, a diffeomorphism
admitting a nontrivial $Tf$-invariant splitting of the tangent
bundle $TM = E^s\oplus E^c \oplus E^u$, such that all unit vectors
$v^\s\in E^\s_x$ ($\s= s, c, u$) with $x\in M$ satisfy:
$$\|T_xfv^s\| < \|T_xfv^c\| < \|T_xfv^u\| $$
for some suitable Riemannian metric. $f$ also must satisfy
that
$\|Tf|_{E^s}\| < 1$ and $\|Tf^{-1}|_{E^u}\| < 1$.  We also want to introduce a stronger type of partial hyperbolicity. We will say that $f$ is {\de absolutely partially hyperbolic}\, if it is partially hyperbolic and $$\|T_xfv^s\| < \|T_yfv^c\| < \|T_zfv^u\| $$ for all $x,y,z\in M$.\par%

 For partially hyperbolic diffeomorphisms, it is a well-known fact that there are foliations $\W^\s$ tangent to the distributions $E^\s$
for $\s=s,u$ . The leaf of $\W^\s$
containing $x$ will be called $W^\s(x)$, for $\s=s,u$. \par
\subsection{Dynamical coherence}In general it is not true that there is a foliation tangent to
$E^c$. It can fail to be true even if $\dim E^c =1$ (see \cite{example}). We shall say that $f$ is {\de dynamically coherent} if there exist invariant foliations $\mathcal{W}^{c\s}$ tangent to $E^{c\s}=E^c \oplus E^\s$  for $\s=s,u$. Note that by taking the intersection of these foliations we obtain an invariant foliation $\mathcal{W}^c$ tangent to $E^c$ that subfoliates $\mathcal{W}^{c\s}$ for $\s=s,u$. Along this paper all partially hyperbolic diffeomorphisms will be dynamically coherent.

\subsection{Accessibility} We shall say that a set $X$ is {\de $\s$-saturated} if it is a union
of leaves of the strong foliations $\W^\s$ for $\s=s$ or $u$. We also say
that $X$ is $su$-saturated if it is both $s$- and $u$-saturated. The
 accessibility class of the point $x\in M$ is the
minimal $su$-saturated set containing $x$.  In case there is some
$x\in M$ whose accessibility class is $M$, then the diffeomorphism
$f$ is said to have the {\de accessibility property}. This is
equivalent to say that any two points of $M$ can be joined by a path
which is piecewise
tangent to $E^s$ or to $E^u$. \par%

\subsection{Entropy} Let $n\in \mathbb{N}$ and $\delta > 0$. A finite subset $E$ is  $(n,\delta)$-separated if for $x,\,y \in E$, $x\neq y$, we
have $\max_{i=0,\dots, n} d(f^n(x),f^n(x))\geq\delta$. Let $$h_n(f, K, \delta)=\sup\{\#E; E\subset K\text{ is }(n,\delta)-\text{separated}\}$$ and $$h(f,K)=\lim_{\delta\rightarrow 0}\limsup_{n\rightarrow \infty}\frac1n \log h_n(f, K, \delta).$$ When $K=M$ call $h(f, M)=h_{top}(f)$ the {\de topological entropy} of $f$. It is a well-known fact that $h(f,K)=0$ if $K$ is a subset of a curve having all its iterates with uniformly bounded length.

Along the paper we shall suppose that the definition of entropy of an $f-$invariant  probability $\mu$,  $h_\mu(f)$, is known among other basic concepts of ergodic theory.

The variational principle states that $\sup\{h_\mu(f); \, \mu \, \text{is} f-\text{invariant} \}=h_{top}(f)$ if $f$ is a continuous map of a compact metric space. We will say that an {\bf ergodic} probability $\mu$ satisfying $h_\mu(f)=h_{top}(f)$ is an {\de entropy maximizing measure}. We call the reader's attention with respect to the fact that, for simplicity of the exposition, we are supposing that entropy maximizing measures are ergodic. In fact if the metric entropy of a measure equals the topological entropy, the ergodic decomposition and the variational principle imply that almost all ergodic components are entropy maximizing. Conversely, the metric entropy of a convex combination of entropy maximizing measures has maximum entropy.  In other words, the measures whose metric entropies equal the topological entropy form a convex set.

\section{The Quotient Dynamics}\label{quotient}
One of the important issues in our work is the dynamics induced in the quotient space $M^\ast:= M/\mathcal{W}^c.$ Let $\pi : M \rightarrow M^\ast$ be the quotient map, $M^\ast$ be equipped with the quotient topology and $f^\ast$ be the dynamics induced on $M^\ast.$ For a general partially hyperbolic diffeomorphism with not necessarily compact central foliation, this quotient space may be even non-Hausdorff. However, when all leaves are compact we hope that this is the case (this result was recently proved by P. Carrasco when the center bundle is one dimensional, see \cite{carrasco}). For three dimensional manifolds we know that the volume of leaves of any one-dimensional foliation $\mathcal{W}$ by compact leaves  is uniformly bounded and every leaf has an arbitrarily small saturated neighborhood \cite{epstein}. Then, the quotient is Hausdorff and, moreover, it is homeomorphic to a closed surface. In other words, $\mathcal{W}$ is a Seifert fibration without boundary.

\subsection{Simple example} Before going into the more general case, let us state a simple example of diffeomorphisms satisfying the hypothesis of our theorem. Let $$ A :=
\begin{pmatrix}
2&1\cr 1&1 \cr
\end{pmatrix}
$$ and $\phi: \mathbb{T}^2 \rightarrow S^1$. Then the Skew product $F: \mathbb{T}^2 \times S^1 \rightarrow \mathbb{T}^2 \times S^1, F(x, \theta) = (A(x), \theta + \phi(x))$ is a partially hyperbolic diffeomorphism with compact central leaves. By standard Hirsch-Pugh-Shub theory, any $C^1-$perturbation of $F$ is partially hyperbolic with compact central leaves and leaf conjugated to $F.$ In particular if $G$ is $C^1-$close to $F$ then $\mathbb{T}^3 /\mathcal{F}^c(G)$ is homeomorphic to $\mathbb{T}^2$ and the dynamics induced in the quotient is conjugate to $A.$

 We mention that by a result of A. Hammerlindl \cite{hammerlindl1} this leaf conjugacy can be obtained not only for perturbations but also for any other absolutely partially hyperbolic diffeomorphism having the same  linear part. He has also announced similar result for the 3-dimensional nilmanifolds. In particular he obtained  that any absolutely partially hyperbolic diffeomorphism of one of these manifolds, that is known to be dynamically coherent \cite{brin-burago-ivanov, parwani}, has compact center leaves.

 The only hypothesis that a priori remains unsatisfied is the accessibility property. But accessibility is an open and dense property in this context (see \cite{nitica-torok, bhhtu}).

 \subsection{Orientability of bundles} From now on we will assume that $M,\,E^u,\, E^s\, E^c$ are orientable and that $f$ preserves the orientation of these bundles. Let us show that these assumptions are not restrictive. Suppose that Theorem \ref{dichotomy} is valid under these conditions. Now, let $g:N\rightarrow N$ be a diffeomorphism satisfying the hypothesis of Theorem \ref{dichotomy}. By taking a finite covering we can obtain a diffeomorphism $\overline g:\overline N\rightarrow \overline N$ such that $\overline N$ and the corresponding invariant bundles are orientable and $h=\overline g^2$ preserves its orientations. Observe that accessibility of $g$ implies accessibility of $\overline g$ and so, accessibility of $h$. Then, $h$ satisfies one of the two possibilities  of the dichotomy given by Theorem \ref{dichotomy}.

  Suppose that $\mu$ is an entropy maximizing measure for $h$ with zero center Lyapunov exponent.  Uniqueness of the entropy maximizing measure of $h$ implies that $\mu$ is $\overline g$-invariant and, of course, it has null center Lyapunov exponent. Moreover, $\mu$ is the unique entropy maximizing measure for $\overline g$.  Thanks to the Rokhlin disintegration  $\mu$ projects onto an $g$-invariant  measure $\nu$. It is not difficult to see that $\nu$ is an entropy maximizing measure of $g$. Uniqueness is a consequence of the uniqueness of the entropy maximizing measure of $\overline g$. Indeed let $\tilde \nu$ be entropy maximizing for $g$ with null center Lyapunov exponent. Then, there exists $\tilde \mu$ a $\overline g$-invariant lift of $\tilde \nu$. On the one hand, the finiteness of the covering gives that $h_{\tilde \nu}=h_{\tilde \mu}$ and $h_{\nu}=h_{\mu}$. On the other hand, since $\tilde \nu$ and $\nu$ are entropy maximizing we have that $h_{\tilde \nu}=h_{\nu}$ leading to  $h_{\mu}=h_{\tilde\mu}$. The uniqueness of the entropy maximizing entropy of $\overline g$ implies $\tilde \mu=\mu$. Bernoullicity of $\mu$ for $h$ implies Bernoullicity of $\mu$ for $\overline g$ and since $(g,\nu)$ is a factor of $(\overline g, \mu)$ it also has the Bernoulli property. Finally if $g$ has an entropy maximizing measure with zero center exponent then, $h$ also has an entropy maximizing measure with zero center exponent. This implies uniqueness for the maximizing measure of $h$ and the previous argument gives the uniqueness for $g$.

Assume  now that $h$ has two entropy maximizing measures, $\mu^+$ and $\mu^-$ with positive and negative center exponent, respectively. Now, $(\mu^\s+\overline g_\#\mu^\s)/2$, with $\s=+,-$, are entropy maximizing measures for $\overline g$ with positive and negative center exponent, respectively. Like in the previous argument, these measures project onto entropy maximizing measures for $g$ that are different because the center exponent of one of them is positive and the other one is negative. This gives the second part of the dichotomy for $g$.

\subsection{$M$ is a nilmanifold} Let $f:M\rightarrow M$ be a partially hyperbolic diffeomorphism with $\text{dim} (M)=3$. Recall that we are assuming that $f$ is dynamically coherent and that the center manifolds are compact i.e. $W^c (x)$ is diffeomorphic to $\mathbb{S}^1$ for every $x\in M$. Moreover, as we have just showed, for our purposes it is enough to consider that $M,\,E^u,\, E^s\, E^c$ are orientable and that $f$ preserves the orientation of these bundles.

$\mathcal{W}^c$ is a foliation by circles of a closed $3$-manifold then, as we have already said,  Epstein's result \cite{epstein} implies that  $\mathcal{W}^c$ is a Seifert foliation. In our case we can obtain more information from the partially hyperbolic structure. Since $W^c(x)$  is diffeomorphic to $\mathbb{S}^1$ we have that $W^{cs}=W^s(W^c(x))$ is an immersed cylinder (observe that $\mathcal{W}^s$ is orientable, this prevents the presence of M\"{o}bius' strips) and then, the foliation $\mathcal{W}^c$ is locally trivial when restricted to $W^{cs}(x))$. Analogously, $\mathcal{W}^c$ is locally trivial when restricted to $W^{cu}(x))$. These facts and the transversality of the foliations $\mathcal{W}^{cs}$ and $\mathcal{W}^{cu}$ imply that $\mathcal{W}^c$ is locally trivial, that is, for every leaf $W^c(x)$ there exist a saturated neighborhood $V_x$ and a homeomorphism $\varphi_x:V_x\rightarrow \mathbb{D}^2\times \mathbb{S}^1$ such that $\varphi_x$ sends leaves of $\mathcal{W}^c$ to circles of the form $\{w\}\times \mathbb{S}^1$. This implies that $\mathcal{W}^c$ is a Seifert fibration (as we mentioned before this was already known for us by Epstein's result) without singular leaves. Moreover, we have the following result:

\begin{theorem}\label{Misseifert}
Let $M$ a closed orientable 3-dimensional manifold and $f:M\rightarrow M$ a partially hyperbolic diffeomorphism. Suppose that $f$ is dynamically coherent with compact center manifolds and that the bundles $E^s, \, E^u, \, E^c$ are orientable. Then,
$\mathcal{W}^c$ is a Seifert fibration without singular leaves over $\mathbb{T}^2$. Morever,\underline{} $f^\ast$ is conjugated to a hyperbolic automorphism of $\mathbb{T}^2$.
\end{theorem}

\begin{proof}
On the one hand, the comments above imply that $\mathcal{W}^c$ is a Seifert fibration without singular leaves.  On the other hand, center stable and center unstable manifolds are cylinders  go via the quotient to lines. Partial hyperbolicity implies that this lines are the stable and unstable ``manifolds" of a hyperbolic homeomorphism. By Franks' results \cite{franks} we obtain that the base surface of the Seifert fibration is a torus and the induced dynamics is conjugated to a hyperbolic automorphism of $\mathbb{T}^2$.
\end{proof}

We observe that, due to the classification of Seifert fibrations,  an $M$ satisfying the conclusions of Theorem \ref{Misseifert} is diffeomorphic to a nilmanifold (including $\mathbb{T}^3$) or, equivalently, to a mapping torus of $\begin{pmatrix}
1&n\cr 0&1 \cr \end{pmatrix}$ with $n\in \mathbb{N}$ (see, for instance, \cite{hatcher}).

As we have mentioned yet, in case that $n\neq 0$ Hammerlindl has announced that he is able to proof that all center curves are compact if $f$ is absolutely partially hyperbolic. Moreover, in \cite{hhu-nilman} it is shown that any partially hyperbolic diffeomorphism of a 3-dimensional nilmanifold has the accessibility property. These two facts lead to the following corollary.

\begin{corollary}
Any absolutely partially hyperbolic diffeomorphism of a 3-dimensional nilmanifold satisfies the dichotomy of Theorem \ref{dichotomy}.
\end{corollary}

\section{Existence and structure of entropy maximizing measures}\label{existence}

Take any invariant measure $\mu$ for $f$ and let $\nu = \mu\circ \pi^{-1}$ (observe that $\nu$ is given by Rokhlin  disintegration). By Ledrappier-Walters variational principle \cite{Francoi-Walters}
$$
 \sup_{\hat{\mu} :\hat{ \mu} \circ \pi^{-1} = \nu } h_{\hat{\mu}}(f) = h_{\nu} (f^\ast) +
 \int_{M^\ast} h(f, \pi^{-1}(y)) d\nu(y).
$$
 Since $\pi^{-1} (y),\, y \in M^\ast$, is a circle and its iterates have bounded length we have that  $h(f, \pi^{-1}(y)) =0,$  that is,  fibers does not contribute to the entropy.  Hence, by the above equality and the well-known fact that $h_{\mu} (f) \geq h_{\nu}(g)$ we conclude that $h_{\mu}(f) = h_{\nu} (g).$  Using the usual variational principle this implies that the topological entropy of $f$ and $f^\ast$ coincide.
 In particular,
 the set of entropy maximizing measures of $f$ coincides with the subset of ergodic measures which projects down to an entropy maximizing measure for $f^\ast.$

In general the existence of entropy maximizing measures is a nontrivial question.
However, by the comments above, in the setting of Theorem \ref{dichotomy}, it is easy to see that entropy maximizing measures always exist. Indeed let $\mathcal{M}$ be the set of probability measures on $M$ and consider
$$ \mathcal{M}_0 := \{ \mu \in \mathcal{M} \colon \pi_*(\mu) = m \} $$ where $m$ is the entropy maximizing measure of $f^\ast$.

It is clear that $\mathcal{M}_0$ is a nonempty convex compact subset of $\mathcal{M}$ and the operator $f_* : \mathcal{M}_0 \rightarrow \mathcal{M}_0$ has a fixed point.

  We finishes this section recalling some properties of entropy maximizing measures of the quotient dynamics. An invariant probability measure $m$ has local product structure if for every $x$ in the support of $m$ there exist measures $m^s, m^u$ on $B^s(x), B^u(x)$ such that $m|_{B(x)}$ is equivalent to $m^s \times m^u.$
Let $g$ be a linear Anosov diffeomorphism on $\mathbb{T}^2$. Then $g$ is topologically mixing and it is well-known that it admits a unique maximizing measure $m$ which is the Lebesgue measure. Let us remind that in general by Bowen and Margulis constructions of maximizing measures, any topologically mixing Anosov diffeomorphism admits a unique maximizing measure which has local product structure.

Now take any dynamics like $f^\ast$ (quotient to central foliation dynamics) which is topologically conjugate to a linear Anosov diffeomorphism. Let $h$ be the conjugacy $ f^\ast \circ  h = h \circ g.$ Then $h_*m$ is the unique maximizing measure of $f^\ast$. Moreover, as $h$ carries the stable (unstable) manifolds of $g$ to stable (unstable) sets of $f^\ast$ we easily conclude that $h_*m$ also has local product structure.

\section{Rigidity and zero Lyapunov exponents }\label{rigidity}

 Here we prove the first part of the theorem. Suppose that there exists an  entropy maximizing measure $\mu$ with $\lambda_c(\mu) =0 $ or, more general, there exist a sequence of entropy maximizing measures such that $\lambda_c(\mu_n) \rightarrow 0$. In both  cases we prove that $(f, \mu)$ is isomorphic to a rotation extension of an Anosov diffeomorphism $A.$
In the second case $\mu$ stands for a weak accumulation point of $\mu_n$ which will also have maximum entropy.

To prove the above claim we apply firstly the invariance principle of Avila-Viana \cite{AV} to conclude that $\mu$ admits a desintegration $[x \rightarrow \mu_x]$ which is $s$-invariant and $u$-invariant and $x \rightarrow \mu_x$ varies continuously with $x$ on $\supp(\pi_{*}(\mu)) = M^\ast$. We state Avila-Viana result adapted to our setting.

\begin{theorem}[Theorem D, \cite{AV}]
Let $f:M\rightarrow M$ be a partially hyperbolic diffeomorphism with one-dimensional compact center leaves. Assume that given $y\in W^\s(x)$ the naturally defined $\s$-holonomy between $W^c(y)$ and $W^c(x)$ is a homeomorphism for $\s=s,\,u$. Let
$(m_k)_k$ be a sequence of $f$-invariant probability measures whose projection
$\mu^\ast$
is  a probability measure that has local product structure. Assume the sequence
converges to some probability measure $\mu$ in the weak$^\ast$ topology
and
$\int |\lambda_c(x)|dm_k(x)\rightarrow 0$ when $k\rightarrow \infty$. Then, $\mu$ admits a disintegration
$\{\mu_{x^\ast}: \, x^\ast \in M^\ast\}$ which is $s$-invariant and $u$-invariant and whose conditional
probabilities $\mu_{x^\ast}$ vary continuously with $x^\ast$
on the support of $\mu^\ast$.
\end{theorem}

 By accessibility and holonomy invariance of $[x \rightarrow \mu_x]$ we conclude that\newline  $\supp(\mu_x) = W^c(x)$ and $\mu_x$ is atom free. Suppose that $\mathcal{W}^c$ is orientable and take and orientation for it. Now using a method as in Avila-Viana-Wilkinson  we get an $S^1-$action on $M$ with commutes with $f$, i.e. $\rho_{\theta}: M \rightarrow M$ such that $\rho_{\theta} \circ f = f \circ \rho_{\theta}, \theta \in S^1.$ We define $\rho$ in the following way: $\rho_\theta(x)=y$ where $y$ is the point in $W^c(x)$ such that the arc of center manifold joining $x$ with $y$ has conditional measure $\mu_x([x,y]_c)=\theta$ (we are identifying $\mathbb{S}^1 $ with $[0,1]$ mod. $1$) and we are measuring in the positive direction.

 \subsection{Rigidity} We will show that this action is conjugated to an action that is isometric on fibers. These will imply that $(f,\mu)$ is conjugated to an isometric (rotation) extension of a Bernoulli shift.

 Our next goal is to prove that $f$ is conjugated to an isometric extension of a hyperbolic automorphism of $\mathbb{T}^2$ if the center Lyapunov exponent of a entropy maximizing measure $\mu$ is 0.

 Recall that we used the invariance principle of Avila-Viana \cite{AV} to conclude that $\mu$ admits a disintegration $[x \mapsto \mu_x]$ which is $s$-invariant and $u$-invariant and $x \mapsto \mu_x$ varies continuously with $x$ on $\supp(\pi_{*}(\mu)) = M^\ast (=M/\mathcal{W}^c)$. Moreover, we have concluded that  $\supp(\mu_x)=W^c(x)$ and the measures $\mu_x$ are atom free. Now we want  to obtain the desired conjugacy. The main difficulty here is the absence of a global section of $\mathcal{W}^c$ (except in the case $M=\mathbb{T}^3$).

 \begin{proposition}\label{conjugacy}
 Let $\mathcal{W}$ be a 1-dimensional foliation by compact leaves of a closed manifold $M$ and $\rho:\mathbb{S}^1\times M\rightarrow M$ an (at least continuous) action such that
 \begin{enumerate}
 \item $\mathcal{W}$ has no singular leaves, that is, every leaf has a neighborhood where the foliation is a product.
 \item The leaves of $\mathcal{W}$ are the orbits of $\rho$.
 \item Given $x\in M$ the map $g\mapsto \rho(g,x)$ is a homeomorphism between $\mathbb{S}^1$ and $W^c(x)$.
 \end{enumerate}
 Then, $\rho$ is conjugated to an action $\overline \rho$ of $\mathbb{S}^1$ in a fiber bundle with fiber $\mathbb{S}^1$ and base $M/\mathcal{W}$ (the principal bundle associated to the fiber bunddle $\mathcal{W}$) Moreover, $\overline \rho$  acts by isometries in the fibers.
 \end{proposition}

 \begin{remark}
 The regularity of $\overline \rho$ depends on the regularity of $\mathcal{W}$ and $\rho$. In the case we are interested in both $\mathcal{W}$ and $\rho$ are H\"{o}lder continuous and then, we obtain the same regularity for $\overline \rho$.
 \end{remark}

  \begin{proof}
  We give an outline of the proof since the arguments are classical in fiber bundle theory (see, for instance, \cite{steenrod}). Let $p:M\rightarrow M/\mathcal{W}$ be the natural projection. Take a covering $\{V_i\}$ of $M/\mathcal{W}$ by balls such that the  $p^{-1}(V_i)$ are trivial. Let $\varphi_i:V_i\times \mathbb{S}^1\rightarrow p^{-1}(V_i)$ be coordinate functions chosen in such a way that $p\varphi_i=\pi$ (where $\pi $ is the projection on the first coordinate of $V_i\times \mathbb{S}^1$) and conjugate the restrictions of $\rho$ to $p^{-1}(V_i)$ with the actions $((x,\theta), \alpha)\mapsto (x,\theta)$ on $V_i\times \mathbb{S}^1$.

  Now, define the new bundle by taking the disjoint union $\bigsqcup_i V_i\times \mathbb{S}^1$ and the quotient by the following relation of equivalence: $(x,\theta)$ with $x\in V_i$ is equivalent to $(\overline x,\overline \theta)$ with $\overline x \in V_j$ if $\varphi_j\varphi_i^{-1}(x,\theta)=(\overline x, \overline \theta)$. This defines our new fiber bundle and the identity map of $\bigsqcup_i V_i\times \mathbb{S}^1$  induces a homeomorphism between $M$ and the new bundle that is not difficult to show that conjugates $\rho$ and $\overline\rho$.
  \end{proof}

 Since $f$ commutes with $\rho$ we obtain, using the conjugacy given by the proposition above, a homeomorphism $\overline f$ conjugated to $f$ that commutes with $\overline \rho$. This implies that $\overline f$ is an isometric extension of the dynamics $f^\ast$ of $f$ on $M/\mathcal{W}^c$. As we have shown before this quotient dynamics is conjugated to an Anosov dynamics of $\mathbb{T}^2$. Let us call $\overline \mu$ to the measure of $M$ formed by the entropy maximizing measure of $f^\ast$ in the base and such that  the  conditional measures along the fibers are Lebesgue. Then, $\overline \mu$ is a rotation extension of a Bernoulli measure.

 Accessibility of $f$ and Brin's theory on compact group extensions imply that $(\overline f, \overline \mu)$ is weakly mixing (see, for instance, \cite[Subsection 2.2]{dolgo}). Rudolph's result \cite{Rudolph} implies that
  $(\overline f, \overline \mu)$ (and then $(f,\mu)$) is Bernoulli.

 \subsection{Uniqueness} It only remains to show that $\mu$ is the unique entropy maximizing measure for $f$. Of course this is equivalent to show that $\overline \mu$ is the unique unique entropy maximizing measure for $\overline f$. Suppose that $\lambda$ is an entropy maximizing measure for $\overline f$. Then, as it was shown before  $\lambda$ projects onto the unique entropy maximizing measure of $f^\ast$ and accessibility implies that the conditional measures along center leaves are supported on the whole $\mathbb{S}^1$, atom free and invariant by rotations (Avila-Viana Invariance Principle again \cite{AV}) This yields that
 the conditional measures are Lebesgue and $\lambda=\overline\mu$.

 \begin{remark}
 In this paper we present a proof for $\dim(M)=3$.  The reader can easily verify that the same proof works for higher dimensions if one asks for the following conditions:
 \begin{enumerate}
 \item The center foliation is one-dimensional, by compact curves and without singular leaves.
 \item The quotient dynamics $f^\ast$ is transitive.
 \end{enumerate}

 Under these hypothesis $f^\ast$ is a transitive hyperbolic homeomorphism and then, it has a unique entropy maximizing measure for that is locally a product of measures of the stable and unstable manifolds (Bowen's measure \cite{bowen,margulis})

 In particular our result remain valid for accessible perturbations of a diffeomorphism of the form $g\times id|_{\mathbb{S}^1}$ where $g$ is a transitive Anosov diffeomorphism of a manifold of any dimension. Recall that accessibility is abundant in this setting (see \cite{bhhtu}).
 \end{remark}

\section{Second part of the dichotomy}\label{parte2}

Now we suppose that there exists $c_0>0$ that for all entropy maximizing measures $|\lambda_c(\mu)| > c_0$. If this were not the case we would have as sequence of entropy maximizing measures $(\mu_n)_n$ weakly$^\ast$ convergent to $\mu$ and such that $\lambda_c(\mu_n)\rightarrow 0$. The arguments of the previous section would imply that $f$ is conjugate to a rotation extension and has a unique measure with null center Lyapunov exponent.

\subsection{More than one measure} Our first step is to show that given an entropy maximizing measure  $\mu^+$ such that $\lambda_c(\mu^+) >0$ there exists a naturally associated measure $\mu^-$ with $\lambda_c(\mu)^-<0$.

\begin{lemma}\label{katomos}
There exists a set $S$ and $k\in \mathbb{N}$ such that $\mu^+(S)=1$ and for every $x\in S$, we have $\# S\cap \pi^{-1}(\pi(x))=k$.
\end{lemma}
\begin{proof}
The lemma  is a corollary of Theorem II of \cite{ruelle-wilkinson}. We observe that since $f\in C^{1+\alpha}$ and the center manifolds also vary continuously in the $C^{1+\alpha }$ norm we are in the conditions to apply the Ruelle-Wilkinson theorem.
\end{proof}

Lemma \ref{katomos} implies that conditional measures of almost every center manifold are supported in $k$ points and thus obtaining  that the weight of each one of these points is $\frac1k$ as a consequence of the ergodicity of $\mu^+$.

We can also suppose that $S$ is contained in the set of Pesin's regular points. Then, each point $x\in S$ has a two-dimensional Pesin unstable manifold $W_P^u(x)$ and we will denote $W_{\lambda_c}(x)$ the intersection of this unstable manifold with the center manifold of $x$, $W_P^u(x)\cap W^c(x)$ the Pesin center manifold of $x$. Since $W^c(x)$ is diffeomorphic to $\mathbb{S}^1$ we have that $W_{\lambda_c}(x)$ is an arc. Call $\tilde x$ to the extreme point of $W_{\lambda_c}(x)$ in the positive direction (recall that we have an orientation for $E^c$). Now we are in condition to define the measure $\mu^-$. We want $\mu^-$ to be an entropy maximizing measure so, it must project onto the same measure as $\mu^+$. Then, it will be enough to define the conditional measures of $\mu^-$ along the center manifolds. Then, given $x\in S$ we will assign weight $\frac1k$ to $\tilde x$.
This defines the conditional measures. The invariance of $\mu^+$ and unstable manifolds imply invariance of $\mu^-$ and ergodicity of $\mu^+$ implies ergodicity of $\mu^-$.

\begin{lemma}
$\lambda_c(\mu^-)<0$.
\end{lemma}
\begin{proof}
First of all observe that $\lambda_c(\mu^-)\neq 0$. Indeed, if it were not the case, the arguments of Section \ref{rigidity} would imply that $\mu^-$ is the unique entropy maximizing measure. Now, if $\lambda_c(\mu^-)>0$ the Pesin unstable manifolds of $\tilde x$ coincide intersects the Pesin unstable manifolds of $x$ contradicting that $\tilde x$ is in the boundary of $W_{\lambda_c}(x)$. This implies the lemma.
\end{proof}

\subsection{Finiteness} We have proved that we have more than one entropy maximizing measure if we have an entropy maximizing measure with nonzero center exponent. All these entropy maximizing measures are trivially shown to be finite extension of Bernoulli shifts (the entropy maximizing  measure of the quotient dynamics is equivalent to a Bernoulli shift).
Then, the only thing we have to prove to finish the proof of the second item of Theorem \ref{dichotomy} is the finiteness of the entropy maximizing measures.

\begin{lemma}\label{convdebil}
Suppose that there is a sequence of entropy maximizing measures $(\mu_n^+)_n$  with $\mu_i^+\neq \mu_j^+$ for $i\neq j$. Moreover, suppose that $\lambda_c(\mu_n^+)>0$ for all $n$ and that the sequence  converges in the weak$^\ast$ topology to a measure $\mu$.
Then, the sequence $(\mu^-_n)_n$ also converges to $\mu$.
\end{lemma}

\begin{proof}
Since the $\mu_n^+$ are infinitely many different measures we have that for almost every center manifold the lengths of the Pesin center manifolds of the points that supports the corresponding conditional measures $\mu_n^+(x)$ go to 0. For this is important to observe that the Birkhoff Theorem implies that the support of these conditional measures do not intersect.  Then, we have that given $\varepsilon>0$ there exists an $N>0$ such that for $n>N$ the length of  the Pesin center manifold of $x$ is less than $\varepsilon $ for every $x$ in a set of center manifold of quotient measure greater than $1-\varepsilon$. This implies that for any continuous $\phi:M\rightarrow \mathbb{R}$ we have that  $\int\phi \;d\mu_n^+-\int\phi\; d\mu_n^-\rightarrow 0$. Then, we obtained that $\mu_n^-\rightarrow \mu$ if $\mu_n^+\rightarrow \mu$.
\end{proof}

 Now let us show that there exists just a finite number of ergodic entropy maximizing measures. That is the compact subset of measures that project onto the entropy maximizing measure of $f^\ast$ is a finite simplex.

 Suppose by contradiction that we have infinitely many entropy maximizing measures. As we have already shown the center Lyapunov exponents of these measures are nonzero. Then, we can take a infinite sequence having the center exponent with the same sign. Suppose that for this sequence the center exponent is positive (if not take the inverse). Since the set of invariant probabilities is sequentially compact we obtain a sequence of measures $(\mu_n^+)_n$ converging to a measure $\mu$ and satisfying  the hypothesis of Lemma \ref{convdebil}. Then, $(\mu_n^-)_n$ also converges to $\mu$.
 On the one hand, as we have observed at the beginning of this section $|\lambda_c(\mu_n^\s)|>c_0, \, \s=+,\,-,$ then, $$\int \log||Df|_{E^c}||\;d\mu=\lim \int \log||Df|_{E^c}||\;d\mu_n^+=\lim \lambda_c(\mu_n^+)\geq c_0.$$

 On the other hand, thanks to the same observation we have
 $$\int \log||Df|_{E^c}||\;d\mu=\lim \int \log||Df|_{E^c}||\;d\mu_n^-=\lim \lambda_c(\mu_n^-)\leq -c_0.$$

 These two inequalities obviously yield to contradiction finishing the proof of the second part of Theorem \ref{dichotomy}.

 \subsection{End of the proof} To finish the proof of Theorem \ref{dichotomy} we only need the following observations: \begin{enumerate}
\item Accessibility is an open and dense property for systems with one-dimensional center (see \cite{didier,bhhtu}).
\item The quotient dynamics is conjugated to a hyperbolic automorphisms so, the center foliation is plaque expansive. Hirsch-Pugh-Shub theory on normally hyperbolic foliations implies that dynamical coherence is an open property.
 \item If we are in the first case of the dichotomy, $f$ has no hyperbolic periodic points. It is easy to see that in this setting (there are periodic leaves diffeomorphic to $\mathbb{S}^1$) the diffeomorphisms having hyperbolic periodic points form an open and dense set.

 \end{enumerate}

\section{Proof of Theorem \ref{mixing} and Corollary \ref{answer-ny}}\label{section.mixing}

\begin{proof}[Proof of Theorem \ref{mixing}]

Let $A:\mathbb{T}^2\rightarrow \mathbb{T}^2$ be a hyperbolic automorphism and consider $F:=A\times \mathbb{S}^1:\mathbb{T}^2\times \mathbb{S}^1\rightarrow \mathbb{T}^2\times \mathbb{S}^1$. Bonatti and D\'{\i}az \cite{bonatti-diaz} have shown that $F$ can be approximated by robustly mixing diffeomorphisms. In other words there exists an open set of mixing partially hyperbolic diffeomorphisms $\mathcal{M}$ such that $F$ belongs to its closure.

The set of accessible diffeomorphisms $\mathcal{A}$ is open and dense in the set of partially hyperbolic diffeomorphisms with one dimensional center. Then, $\mathcal{N}=\mathcal{M}\cap \mathcal{A}\neq \emptyset$. The diffeomorphism of $\mathcal{N}$ satisfy the hypothesis of Theorem \ref{dichotomy} which implies that  the diffeomorphisms of an open end dense subset of $\mathcal{N}$ satisfy the second item of the dichotomy. This proves Theorem \ref{mixing}.

\end{proof}

\begin{proof}[Proof of Corollary \ref{answer-ny}]
If we neglect a set of null measure of fibers the projection $\pi$ induces a conjugation between the support of an entropy maximizing measure and a linear automorphisms of the torus times a finite permutation.  \end{proof}

\section{Questions and Remarks}\label{questions}

\subsection{Noncompact center leaves} As we mentioned in the introduction there are some advances in the case that the diffeomorphism is homotopic to Anosov (see \cite{bfsv, ures}). Ures' result uses that the invariant foliations are  quasi-isometric. Then, the following questions remain open:
\begin{question}
Let $f$ be a partially hyperbolic diffeomorphism of $\mathbb{T}^3$ homotopic to a hyperbolic automorphism. Is it always true that $f$ has a unique entropy maximizing measure? Or, presumingly equivalent in this case, is its topological entropy equal to the entropy of its linear part?
\end{question}

A similar question, in the sense of generalizing Theorem \ref{dichotomy}, remains open for the case that $f$ is homotopic to a hyperbolic diffeomorphism of $\mathbb{T}^2$ times the identity of $\mathbb{S}^1$ (or to a automorphism of a nilmanifold) in case that $f$ is dynamically incoherent.
Observe that the dynamically incoherent examples of \cite{example} are Axiom A and satisfy the conclusions of the second part of Theorem \ref{dichotomy}.

A different direction of future development is the case of perturbations of time one map of an Anosov flows. One of the difficulties in this case seems to be the fact that the entropy is not constant when the time varies. An interesting question would be the following:

\begin{question}
Is the topological entropy generically locally constant in a neighborhood of the time one map of an Anosov flow?
\end{question}

In \cite{hsx}, Y. Hua, R. Saghin and Z. Xia gave some topological conditions to obtain that the entropy is locally constant but these conditions are not valid in the neighborhood of the preceding question.

\subsection{Symplectic case}
There are many interesting problems in symplectic context (like $n$-body problem) where a partially hyperbolic invariant subset appears. By definition, a partially hyperbolic subset is a compact invariant subset with a partially hyperbolic decomposition of the tangent space satisfying similar conditions to the partially hyperbolic diffeomorphisms. As a simple and relevant model consider a symplectic manifold $N$ and $g : N \rightarrow N$ a symplectomorphism. Now take $f: M \rightarrow M$ any diffeomorphism with a transitive hyperbolic invariant subset $\Lambda \varsubsetneqq  M.$ Then $\Lambda \times N$ is partially hyperbolic subset for $ F:= f \times g$. By Hirsch-Pugh-Shub theory this subset has continuation after $C^1-$perturbations of $F.$ (See \cite{NP} for mixing properties of these kind of subsets.)
Let $N$ to be a surface with a symplectic form.  Then either both central Lyapunov exponents of $F$ (for any invariant measure) vanish or the sum of the central Lyapunov exponents is equal to zero.

 The invariant subsets under consideration have zero Lebesgue measure.
The class of entropy maximizing measures is  a natural object of study for the dynamics of such subsets. Tahzibi and Nassiri conjectured  that $F$ is in the closure of an open set $U$ such that any $G \in U$ has a partially hyperbolic set satisfying the following dichotomy: Either there exists an ergodic entropy maximizing measure with zero Lyapunov exponents (and in this case it is the unique maximizing measure) or there are finitely many (more than one) ergodic maximizing measures all of them with non zero Lyapunov exponents.

 One of the difficulties to deal with this conjecture is the lack of accessibility property in this context.   However some strong transitivity property for the holonomy (stable and unstable) semi group (\cite{TN}) is helpful to approach the problem.

\subsection{Number of maximizing measures in Theorem \ref{dichotomy}}

In general, we can obtain any even number of entropy maximizing measures just taking the product of an Anosov diffeomorphism  of $\mathbb{T}^2$ and a Morse-Smale diffeomorphism of $\mathbb{S}^1$ (an odd number of measures can be obtained if, for instance, we allow that $f$ does not preserve the center orientation). Then, the relevant questions are in case $f$ is topologically mixing or transitive.

\begin{question}
There are topologically mixing $f$ satisfying the hypothesis of Theorem \ref{dichotomy} and with more than 2 entropy maximizing measures? Is it true that generically $f$ has just 2 such measures?
\end{question}

\subsection{Conservative case}

According to \cite{AVW} if $f$ satisfies the hypothesis of Theorem \ref{dichotomy} and is conservative  (it preserves a smooth measure $m$) we have three possibilities (recall that accessibility implies $K$):

\begin{enumerate}
 \item $m$ has nonzero center exponents  and then, it is Bernoulli by well-known Pesin's results.
  \item The center exponent vanishes and the center foliation is absolutely continuous. It is proved that $f$ is smoothly conjugated to a rotation extension of a Bernoulli system and then, it is Bernoulli thanks to Rudolph's result \cite{Rudolph}.
  \item The center exponent vanishes but the center foliation is not absolutely continuous. The conditional measures along the center foliation are atomic and the projection of $m$ is not locally a product in the quotient dynamics.

\end{enumerate}

In the third case it is unknown if $m$ is Bernoulli. A negative answer would be very interesting because there are no example of a conservative $K$-diffeomorphism that is not Bernoulli in dimension 3.

\end{document}